\documentclass[12pt]{amsart}

\usepackage[margin=1.15in]{geometry}
\usepackage[T1]{fontenc}
\usepackage[utf8]{inputenc}
\usepackage{lmodern}
\usepackage{microtype}
\usepackage{amsmath,amssymb,amsthm,mathtools,mathrsfs}
\usepackage{enumitem}
\usepackage[dvipsnames]{xcolor}
\usepackage[colorlinks=true,linkcolor=MidnightBlue,citecolor=MidnightBlue,urlcolor=MidnightBlue]{hyperref}
\usepackage{aliascnt}
\usepackage[nameinlink,capitalize]{cleveref}
\usepackage{verbatim}

\allowdisplaybreaks[3]
\usepackage[nodisplayskipstretch]{setspace}
\newtheorem{theorem}{Theorem}[section]
\newaliascnt{proposition}{theorem}
\newtheorem{proposition}[proposition]{Proposition}
\aliascntresetthe{proposition}
\newaliascnt{lemma}{theorem}
\newtheorem{lemma}[lemma]{Lemma}
\aliascntresetthe{lemma}
\newaliascnt{corollary}{theorem}
\newtheorem{corollary}[corollary]{Corollary}
\aliascntresetthe{corollary}
\newaliascnt{definition}{theorem}
\newtheorem{definition}[definition]{Definition}
\aliascntresetthe{definition}
\newaliascnt{remark}{theorem}
\newtheorem{remark}[remark]{Remark}
\aliascntresetthe{remark}
\newaliascnt{convention}{theorem}
\newtheorem{convention}[convention]{Convention}
\aliascntresetthe{convention}
\newaliascnt{example}{theorem}
\newtheorem{example}[example]{Example}
\aliascntresetthe{example}

\crefname{theorem}{Theorem}{Theorems}
\crefname{proposition}{Proposition}{Propositions}
\crefname{lemma}{Lemma}{Lemmas}
\crefname{corollary}{Corollary}{Corollaries}
\crefname{definition}{Definition}{Definitions}
\crefname{remark}{Remark}{Remarks}
\crefname{convention}{Convention}{Conventions}
\crefname{example}{Example}{Examples}

\DeclareMathOperator{\Mar}{Mar}
\DeclareMathOperator{\codim}{codim}

\newcommand{\Mbar}{\overline{\mathcal M}}
\newcommand{\be}{\mathbf e}
\newcommand{\bx}{\mathbf x}
\newcommand{\by}{\mathbf y}
\newcommand{\bz}{\mathbf z}

\newcommand{\bt}{\mathbf t}
\newcommand{\bI}{\mathbf I}

\newcommand{\WK}{\mathrm{WK}}

\title[Virasoro Recursions for Simplicially Stable Curves]{Virasoro Recursions for Simplicially Stable Curves with Colliding Markings}
\author{You-Cheng Chou}
\address{June E Huh Center for Mathematical Challenges\\ Korea Institute for Advanced Study\\ 85 Hoegi-ro, Dongdaemun-gu\\ Seoul 02455\\ Republic of Korea}
\email{bensonchou@kias.re.kr}

\author{Hsian-Hua Tseng}
\address{Department of Mathematics\\ The Ohio State University\\ 100 Math Tower, 231 West 18th Avenue\\ Columbus, OH 43210\\ USA}
\email{hhtseng@math.ohio-state.edu}

\date{\today}

\begin{document}


\begin{abstract}
We study descendant integrals on moduli stacks of simplicially stable curves.  For each finite collision complex, we factor the reduction morphism into elementary wall crossings and obtain a reconstruction formula in terms of ordinary Witten--Kontsevich correlators.  
We use the formula to derive corrected Virasoro recursions. We then package all finite-support theories in a completed disjoint-support Fock space, where the extended potential is a cumulant translate of the Witten--Kontsevich potential.  
\end{abstract}

\maketitle
\tableofcontents

\section{Introduction}

The Witten--Kontsevich theorem identifies the descendant potential of the Deligne--Mumford spaces \(\Mbar_{g,n}\) with the distinguished solution of the KdV hierarchy \cite{Witten,Kontsevich}.  Chou and Lee \cite{CL} extended the Virasoro formalism of the Witten-Kontsevich theorem, due to \cite{KS}, to Hassett spaces of weighted pointed curves \cite{Hassett}. The main feature of Hassett spaces is that marked points are allowed to collide according to weights. Blankers and Bozlee subsequently showed that the essential datum governing collisions of marked points is not a weight vector, but a simplicial complex \cite{BB}.  Their spaces \(\Mbar_{g,\mathcal{K}}\) include all Hassett spaces and many modular compactifications that cannot arise from weighted stability.

Thus it is natural to consider integrals of $\psi$ classes on these more general moduli spaces $\Mbar_{g,\mathcal{K}}$.

This note has three purposes.  First, it proves the reconstruction formula for \(\psi\)-integrals on \(\Mbar_{g,\mathcal{K}}\) by an elementary-wall-crossing induction. This is Theorem \ref{thm:reconstruction}. Second, it transfers the ordinary Virasoro recursion to a fixed finite complex \(\mathcal{K}\), including the exceptional terms that are invisible if one writes only the generic recursion. The results are Proposition \ref{prop:ordinary-corrected}, Theorems \ref{thm:heavy-corrected} and \ref{thm:main-corrected}.  Third, it gives a consistent global formalism.  An arbitrary label set, interpreted coefficientwise through finite supports, allows arbitrarily many light markings, but ordinary partial derivatives are incompatible with the rule that overlapping label supports multiply to $0$.  We therefore work in a completed disjoint-union algebra and use ordinary derivatives only in repeatable heavy variables. See Section \ref{sec:fock} for more details.


\subsection*{Acknowledgment}
The idea of this paper took shape when the authors learned about \cite{BB}.  As part of an experiment concerning the use of generative AI in mathematical research, an earlier draft was produced through interactions with GPT-5.5 Pro.  The present revision corrects and reorganizes that draft.

\section{Simplicially stable curves and correlators}\label{sec:curves}
This section collects some definitions, constructions, and basic properties of moduli stacks of simplicially stable curves.

\subsection{Collision complexes and \(\mathcal{K}\)-stability}

Let \(S\) be a nonempty finite set.  If \((C;(p_i)_{i\in S})\) is an \(S\)-pointed curve and \(x\in C\) is a closed point, define
\[
  \Mar(x):=\{i\in S:p_i=x\}.
\]
For a subcurve \(Z\subset C\), similarly define \(\Mar(Z):=\{i\in S:p_i\in Z\}\).

A simplicial complex \(\mathcal{K}\) on \(S\) is a downward-closed collection of subsets of \(S\) containing every singleton.  A partition of \(S\) is a \(\mathcal{K}\)-partition if each block belongs to \(\mathcal{K}\).  When \(g=0\), we assume that \(\mathcal{K}\) is \emph{at least triparted}, meaning that every \(\mathcal{K}\)-partition has at least three blocks.

\begin{convention} \label{conv:stable-complex}
We call a simplicial complex $\mathcal{K}$ \emph{stable in genus} $g$ if $g>0$, or if $g=0$ and $\mathcal{K}$ is at least triparted.
\end{convention} 

\begin{definition}[{\cite[Definition 2.12]{BB}}]\label{def:Kstable}
An \(S\)-pointed curve \((C;(p_i)_{i\in S})\) of genus \(g\) is \emph{\(\mathcal{K}\)-stable} if:
\begin{enumerate}[label=(K\arabic*)]
\item \(C\) has at worst nodal singularities;
\item every marking is a smooth point;
\item \(\Mar(x)\in \mathcal{K}\) for each smooth point \(x\in C\);
\item \(\Mar(Z)\notin \mathcal{K}\) for every rational tail \(Z\subset C\);
\item the normalization of every rational irreducible component has at least three distinct special points.
\end{enumerate}
\end{definition}

Blankers and Bozlee prove the following results about the moduli \(\Mbar_{g,\mathcal{K}}\) of genus $g$ $\mathcal{K}$-stable curves.
\begin{enumerate}
    \item (\cite[Theorem 4.18]{BB}) $\Mbar_{g,\mathcal{K}}$ is a smooth proper Deligne--Mumford stack containing $\mathcal{M}_{g,S}$.
    The collision complex of $\Mbar_{g,\mathcal{K}}$, in the sense of \cite[Definition~2.6]{BB}, is exactly \(\mathcal{K}\).
    
    \item (\cite[Theorem 4.19]{BB}) Every nodal modular compactification of $\mathcal{M}_{g,S}$ with smooth colliding markings is obtained in this way.  
    
    \item (\cite[Theorem 6.3]{BB}) There is a reduction morphism
\[
  \rho_{\mathcal{K}}:\Mbar_{g,S}\longrightarrow \Mbar_{g,\mathcal{K}}
\]
contracting precisely the rational tails whose marking sets lie in \(\mathcal{K}\).
\end{enumerate}

\begin{remark}
The {\em collision complex} associated to a Hassett vector \(A=(a_i)_{i\in S}\)  is
\[
  \mathcal{K}_A=\left\{I\subseteq S:\sum_{i\in I}a_i\le 1\right\}.
\]
By \cite[Proposition 4.20]{BB}, $\Mbar_{g, \mathcal{K}_A}$ is isomorphic to the Hassett space $\Mbar_{g,A}$. Complexes of the form $\mathcal{K}_A$ are exactly the threshold complexes \cite[Corollary 4.22]{BB}.
\end{remark}

\begin{definition}[Heavy vertices]
If \(H\) is a finite set disjoint from the vertex set of a complex \(\mathcal{K}'\), let
\[
  \mathsf D_H:=\{\emptyset\}\cup\{\{h\}:h\in H\}
\]
be the discrete complex on \(H\), and let $$\mathsf D_H\sqcup\mathcal{K}'$$ 
denote the complex on \( H \sqcup \operatorname{Vert}(\mathcal{K}')\) whose nonempty faces are the singletons \(\{h\}\) and the nonempty faces of \(\mathcal{K}'\). Thus the vertices in H are isolated: the corresponding markings cannot collide with one another or with the markings indexed by vertices of $\mathcal{K}'$. We call the vertices in H \emph{heavy}.
\end{definition}

\subsection{The elementary reduction geometry}

The wall-crossing argument below uses a reduction between two simplicial compactifications.  Since Blankers--Bozlee state their universal reduction from the ordinary stable curve moduli, we record explicitly the elementary factorization and the boundary geometry needed below.

\begin{proposition}[Elementary reduction geometry]\label{prop:elementary-geometry}
Let $S$ be a nonempty finite set, let $H$ be a finite set of heavy labels, and let
\[
  \mathcal{K}^-\subset \mathcal{K}^+
\]
be simplicial complexes on \(S\) such that
\[
  \mathcal{K}^+=\mathcal{K}^-\cup\{I\},
\]
where \(I\in \mathcal{K}^+\) is a maximal face with \(|I|\ge2\).  Assume that the augmented source and target types are stable in genus $g$ in the sense of Convention~\ref{conv:stable-complex}. Then:
\begin{enumerate}[label=(\roman*)]
\item there is a canonical proper birational morphism
\[
 \pi_I:\Mbar_{g,\mathsf D_H\sqcup\mathcal{K}^-}
       \longrightarrow
       \Mbar_{g,\mathsf D_H\sqcup\mathcal{K}^+}
\]
which contracts exactly the rational tails $Z$ with $\Mar(Z)=I$, and the Blankers--Bozlee reductions satisfy
\[
 \rho_{\mathsf D_H\sqcup\mathcal{K}^+}
 =\pi_I\circ\rho_{\mathsf D_H\sqcup\mathcal{K}^-};
\]
\item let $\Delta_I$ denote the exceptional divisor parametrizing curves with an I-marked rational tail. Then
\[
\Delta_I\cong
 \Mbar_{0,\mathsf D_{\{\bullet_t\}}\sqcup(\mathcal{K}^-|_I)}
 \times
 \Mbar_{g,\mathsf D_{H\cup\{\bullet_m\}}\sqcup(\mathcal{K}^-|_{S\setminus I})};
\]
under this identification, $\pi_I|_{\Delta_I}$ is the projection to the second factor followed by the morphism replacing the heavy marking $\bullet_m$ with the collision of the markings in $I$;
\item if \(L_a^\pm\) denotes the cotangent line at a labeled section, then
\[
 L_i^-\cong\pi_I^*L_i^+\otimes\mathcal O(\Delta_I),
 \quad i\in I;
 \qquad
 L_j^-\cong\pi_I^*L_j^+,
 \quad j\in(S\setminus I)\sqcup H,
\]
and
\[
 \mathcal O_{\Delta_I}(\Delta_I)
 \cong
 (L_{\bullet_t}^{\rm tail})^\vee
 \boxtimes
 (L_{\bullet_m}^{\rm main})^\vee.
\]
\end{enumerate}
\end{proposition}

\begin{proof}
By \cite[Proposition~4.16 and Sections~5--6]{BB}, the augmented complexes
\(\mathsf D_H\sqcup\mathcal K^\pm\) correspond to nested extremal assignments.
The only additional components selected by the target assignment are the rational
tails whose marking set is \(I\).  Functoriality of the contraction construction
therefore gives a canonical factorization
\[
\rho_{\mathsf D_H\sqcup\mathcal K^+}
=
\pi_I\circ\rho_{\mathsf D_H\sqcup\mathcal K^-}.
\]
The morphism \(\pi_I\) is proper and is an isomorphism over the common dense open
locus \(\mathcal M_{g,S\sqcup H}\); hence it is birational.

Since every proper subset of \(I\) belongs to \(\mathcal K^-\), an \(I\)-marked
rational tail in a \((\mathsf D_H\sqcup\mathcal K^-)\)-stable curve cannot contain
a proper rational subtail.  It is therefore irreducible.  The standard gluing
description of this boundary gives
\[
\Delta_I\cong
\Mbar_{0,\mathsf D_{\{\bullet_t\}}\sqcup(\mathcal K^-|_I)}
\times
\Mbar_{g,\mathsf D_{H\cup\{\bullet_m\}}
\sqcup(\mathcal K^-|_{S\setminus I})}.
\]
Under this identification, \(\pi_I|_{\Delta_I}\) is the projection to the second
factor followed by replacing \(\bullet_m\) with the collision of the markings in
\(I\).

Finally, the normal bundle of a separating boundary divisor is the tensor product
of the dual cotangent lines at the two branches of the node.  The standard local
calculation for the contraction of a rational tail gives
\[
L_i^-\cong\pi_I^*L_i^+\otimes\mathcal O(\Delta_I),
\quad i\in I;
\qquad
L_j^-\cong\pi_I^*L_j^+,
\quad j\notin I,
\]
together with
\[
\mathcal O_{\Delta_I}(\Delta_I)
\cong
(L_{\bullet_t}^{\rm tail})^\vee
\boxtimes
(L_{\bullet_m}^{\rm main})^\vee;
\]
compare \cite[Lemma~1.3]{CL}.
\end{proof}

\section{Wall-crossing and reconstruction of correlators}

In this section we define and study correlators on $\Mbar_{g,\mathcal K}$.
For ordinary correlators, we use the standard convention
\[
\left\langle \tau_{k_1}\cdots\tau_{k_n}\right\rangle_g
:=
\int_{\Mbar_{g,n}}
\prod_{\alpha=1}^n \psi_\alpha^{k_\alpha}
\]
when $2g-2+n>0$, and $0$ otherwise. We also set an ordinary
correlator equal to $0$ if any descendant exponent is negative. For example,
\[
\left\langle\tau_0^3\right\rangle_0=1,
\qquad
\left\langle\tau_1\right\rangle_1=\frac{1}{24}.
\]

\begin{definition}[Quotient complexes]\label{def:quot_cpx}
Let $\mathcal P=\{I_1,\ldots,I_m\}$ be a collection of pairwise disjoint
nonempty faces of $\mathcal K$. Define the \emph{quotient complex}
\[
\mathcal K/\mathcal P
:=
\left\{
\mathcal J\subseteq\mathcal P:
\bigcup_{I\in\mathcal J} I\in\mathcal K
\right\}
\]
on the vertex set $\mathcal P$.
\end{definition}

Throughout this section, whenever pairwise disjoint nonempty faces
\(I_1,\ldots,I_m\) and a finite set
\(H=\{h_1,\ldots,h_n\}\) of heavy vertices are given, we write
\[
\mathcal P=\{I_1,\ldots,I_m\}
\]
and use the associated augmented quotient type
\[
\mathsf D_H\sqcup(\mathcal K/\mathcal P).
\]

Let $e_1,\ldots,e_m$ be nonnegative integers. If $\pi$ is a partition of
$\{I_1,\ldots,I_m\}$ and $U\in\pi$ is a block, set
\[
I_U:=\bigcup_{I_a\in U} I_a,
\qquad
e_U^\sharp:=\sum_{I_a\in U}e_a-|U|+1.
\]
The partition $\pi$ is called \emph{$\mathcal K$-admissible} if $I_U\in\mathcal K$ for every
block $U\in\pi$. Denote by
$$\Pi_{\mathcal K}(I_1,\ldots,I_m)$$ the set of all
$\mathcal K$-admissible partitions of $\{I_1,\ldots,I_m\}$.

When the light faces are the singletons indexed by \(S\), we identify
partitions of \(\{\{i\}:i\in S\}\) with partitions of \(S\). We write
\(\Pi_{\mathcal K}(S)\) for the set of \(\mathcal K\)-admissible
partitions of \(S\), and for every nonempty subset \(I\subseteq S\), set
\[
e_I^\sharp:=\sum_{i\in I}e_i-|I|+1.
\]

\begin{definition}[Reconstructed correlator]\label{def:formal-corr}
For a finite set $H=\{h_1,\ldots,h_n\}$ of heavy vertices, define
\[
\left\langle
\prod_{\alpha=1}^n\tau_{k_\alpha}
\prod_{a=1}^m\tau_{e_a;I_a}
\right\rangle_g^{\mathcal K,\mathrm{rec}}
:=
\sum_{\pi\in\Pi_{\mathcal K}(I_1,\ldots,I_m)}
(-1)^{m-|\pi|}
\left\langle
\prod_{\alpha=1}^n\tau_{k_\alpha}
\prod_{U\in\pi}\tau_{e_U^\sharp}
\right\rangle_g.
\]
Here the right-hand side is interpreted using the ordinary correlators
defined above.
\end{definition}

\begin{remark}
The reconstructed correlator is defined without requiring the augmented
quotient type
\[
\mathsf D_H\sqcup(\mathcal K/\mathcal P)
\]
to be stable in genus $g$. In particular, when $g=0$, the quotient complex
$\mathcal K/\mathcal P$ need not be at least triparted. For example, for
three light vertices, even if $I_1\cup I_2\in\mathcal K$, so that
$\mathcal K/\mathcal P$ is not at least triparted, we still have
\[
\left\langle
\tau_{0;I_1}\tau_{0;I_2}\tau_{0;I_3}
\right\rangle_0^{\mathcal K,\mathrm{rec}}
=
\left\langle\tau_0^3\right\rangle_0
=
1.
\]
Indeed, every nontrivial block contributes a negative descendant exponent
and hence vanishes by convention. Retaining these formal, non-geometric
correlators is essential for formulating a closed system of Virasoro
constraints.
\end{remark}

\begin{definition}[Geometric correlator]\label{def:geometric-heavy}
Assume that
\[
\mathsf D_H\sqcup(\mathcal K/\mathcal P)
\]
is stable in genus $g$. Define
\[
\left\langle
\prod_{\alpha=1}^n\tau_{k_\alpha}
\prod_{a=1}^m\tau_{e_a;I_a}
\right\rangle_g^{\mathcal K,\mathrm{geom}}
:=
\int_{\Mbar_{g,\mathsf D_H\sqcup(\mathcal K/\mathcal P)}}
\prod_{\alpha=1}^n\psi_{h_\alpha}^{k_\alpha}
\prod_{a=1}^m\psi_{I_a}^{e_a}.
\]
\end{definition}

In the remainder of this section, we prove that the reconstructed and geometric correlators agree in the stable range.

First we have the following wall-crossing result.

\begin{lemma}[Geometric elementary wall-crossing]\label{lem:wall}
In the situation of Proposition \ref{prop:elementary-geometry},
we have
\[
\begin{aligned}
 &\left\langle
   \prod_{\alpha=1}^n\tau_{k_\alpha}
   \prod_{i\in S}\tau_{e_i;i}
  \right\rangle_g^{\mathcal{K}^+,\mathrm{geom}}
 \\
 =&\left\langle
   \prod_{\alpha=1}^n\tau_{k_\alpha}
   \prod_{i\in S}\tau_{e_i;i}
  \right\rangle_g^{\mathcal{K}^-,\mathrm{geom}}
 +(-1)^{|I|-1}
 \int_{\Mbar_{g,\mathsf D_{H\cup\{\bullet\}}\sqcup
                    (\mathcal{K}^-|_{S\setminus I})}}
 \Big(\prod_{\alpha=1}^n\psi_{h_\alpha}^{k_\alpha} \Big)
 \psi_\bullet^{e_I^\sharp}
 \prod_{j\in S\setminus I}\psi_j^{e_j}.
\end{aligned}
\]
The last integral is interpreted as $0$ when \(e_I^\sharp<0\).
\end{lemma}

\begin{proof}
Let \(E=\sum_{i\in I}e_i\) and \(r=|I|\).  By Proposition \ref{prop:elementary-geometry},
\[
  \psi_i^- =\pi_I^*\psi_i^+ +\Delta_I,
  \quad i\in I,
\]
while every other cotangent line class pulls back unchanged.  In any term containing \(\Delta_I\), the restrictions of the classes \(\pi_I^*\psi_i^+\), \(i\in I\), all restrict to the main-branch cotangent line class \(\psi_{\bullet_m}\).  The self-intersection formula is
\[
  \Delta_I^q
  =\iota_*
   \bigl(-\psi_{\bullet_t}-\psi_{\bullet_m}\bigr)^{q-1},
  \qquad q\ge1,
\]
where \(\iota:\Delta_I\hookrightarrow\Mbar_{g,\mathsf D_H\sqcup\mathcal{K}^-}\).

The tail factor has dimension \(r-2\), and
\[
 \int_{\Mbar_{0,\mathsf D_{\{\bullet_t\}}\sqcup(\mathcal{K}^-|_I)}}
 \psi_{\bullet_t}^{r-2}=1.
\]
This equality is geometric and independent of reconstruction.  Indeed, the Blankers--Bozlee reduction
\[
 \Mbar_{0,\{\bullet_t\}\sqcup I}
 \longrightarrow
 \Mbar_{0,\mathsf D_{\{\bullet_t\}}\sqcup(\mathcal{K}^-|_I)}
\]
is proper birational, and no contracted tail contains the isolated marking \(\bullet_t\); hence the heavy cotangent line class pulls back unchanged.  The integral is therefore the ordinary value
\(\int_{\Mbar_{0,r+1}}\psi_{\bullet_t}^{r-2}=1\).

Expanding the factors
\[
    \prod_{i \in I} ( \pi_I^* \psi_i^+ + \Delta_I )^{e_i},
\]
the total contribution supported on \(\Delta_I\) is the main-factor integral in the statement multiplied by
\[
  \sum_{q=1}^{E}
  \binom{E}{q}\binom{q-1}{r-2}(-1)^{q-1}.
\]
If \(e_I^\sharp<0\), both this contribution and the displayed main-factor integral vanish.  Otherwise \(E\ge r-1\), and coefficient extraction from
\[
 \sum_{q=1}^{E}(-1)^{q-1}\binom Eq(1+z)^{q-1}
 =\frac{1-(-z)^E}{1+z}
\]
gives
\[
  \sum_{q=1}^{E}
  \binom{E}{q}\binom{q-1}{r-2}(-1)^{q-1}
  =(-1)^r.
\]
Thus the source integral equals the target integral plus \((-1)^r\) times the main-factor integral.  Rearranging proves the formula.
\end{proof}

Two kinds of correlators defined in Definitions \ref{def:formal-corr} and \ref{def:geometric-heavy} are related:

\begin{theorem}[Reconstruction]\label{thm:reconstruction}
If the augmented quotient type
\[
  \mathsf D_H\sqcup(\mathcal K/\mathcal P)
\]
is stable in genus \(g\), then
\[
 {
 \left\langle
  \prod_{\alpha=1}^n\tau_{k_\alpha}
  \prod_{a=1}^m\tau_{e_a;I_a}
 \right\rangle_g^{\mathcal{K},\mathrm{geom}}
 =
 \left\langle
  \prod_{\alpha=1}^n\tau_{k_\alpha}
  \prod_{a=1}^m\tau_{e_a;I_a}
 \right\rangle_g^{\mathcal{K},\mathrm{rec}}.}
\]
In particular, for singleton light labels,
\[
 \left\langle
  \prod_{i\in S}\tau_{e_i;i}
 \right\rangle_g^{\mathcal{K},\mathrm{geom}}
 =
 \sum_{p\in\Pi_{\mathcal{K}}(S)}(-1)^{|S|-|p|}
 \left\langle\prod_{I\in p}\tau_{e_I^\sharp}\right\rangle_g.
\]
\end{theorem}

\begin{proof}
We use a simultaneous two-level induction.  The outer induction is on the total number
\[
  N=m+n
\]
of vertices of the augmented quotient complex, heavy vertices included.  For a fixed \(N\), the inner induction is on the number of non-singleton faces of the light quotient complex.  The cases \(N=0\) and \(N=1\) are immediate, and for every \(N\) the discrete light complex is the ordinary pointed moduli space, so the formula is the identity.

Fix \(N\), assume the theorem for all stable augmented complexes with fewer than \(N\) vertices, and let \(L\) be a light quotient complex with \(m\) vertices.  Order its non-singleton faces
\[
  J_1,\ldots,J_r
\]
by nondecreasing cardinality, and let \(L_j\) be obtained from the discrete complex by adjoining \(J_1,\ldots,J_j\).  Every proper face of \(J_j\) is already in \(L_{j-1}\), while no strict superset of \(J_j\) is in \(L_j\); hence \(J_j\) is maximal in \(L_j\).  If \(g=0\), every intermediate augmented complex is at least triparted because it is a subcomplex of the final one.

Assume the reconstruction formula has been established for \(L_{j-1}\).  Apply the geometric wall-crossing formula of Lemma \ref{lem:wall} to the face \(J_j\).  Its main factor has
\[
  N-|J_j|+1<N
\]
marked vertices, because the whole block \(J_j\) is replaced by one isolated attaching marking.  The outer induction therefore identifies that \emph{geometric} main-factor integral with its reconstructed partition sum.  This is the step that removes any circular use of reconstruction in the wall-crossing lemma.

The reconstructed main factor runs over the admissible partitions \(q\) of the remaining light vertices; the new attaching vertex is isolated and hence remains a singleton.  Multiplication by the wall sign gives
\[
 (-1)^{|J_j|-1}(-1)^{m-|J_j|-|q|}
 =(-1)^{m-(|q|+1)}.
\]
These are exactly the signs of the admissible partitions for \(L_j\) that were not admissible for \(L_{j-1}\), namely the partitions having \(J_j\) as one block.  This proves the inner induction.  At \(j=r\) it proves the theorem for \(L\), and hence completes the outer induction.
\end{proof}

\begin{remark}
The simultaneous induction controls nested boundary strata.  A direct pullback expansion on \(\Mbar_{g,S}\) contains intersections such as
\(D_{\{1,2\}}\cap D_{\{1,2,3\}}\).  In the proof above, the smaller face is processed at an earlier inner wall, while the main factor at the later wall is handled by the outer induction on the number of marked vertices.  No nested stratum is discarded or silently replaced by a collection of disjoint tails.
\end{remark}

\section{Corrected Virasoro recursions for a fixed complex}\label{sec:recursions}

This section establishes Virasoro recursions for the correlators introduced above. Throughout this section, a correlator with superscript \(\mathcal K\) means the reconstructed correlator of Definition \ref{def:formal-corr}. By Theorem \ref{thm:reconstruction}, it agrees with the corresponding geometric correlator whenever the augmented quotient type is stable.


For \(k\ge0\) and \(e\ge0\), define
\[
  h_{k;e}:=\frac{(2k+2e+1)!!}{(2e-1)!!}.
\]
For the string index, set
\[
  h_{-1;e}:=1,\quad e\ge0,
\]
and, for every \(k\ge-1\), set \(h_{k;e}=0\) when \(e<0\).  If \(\be=(e_1,\ldots,e_m)\), define
\[
  h_{k;\be}:=
  \sum_{\emptyset\ne J\subseteq[m]}(-1)^{|J|-1}
  h_{k;\,\sum_{j\in J}e_j-|J|+1}.
\]

\begin{lemma}[Möbius identity]\label{lem:mobius}
For every nonempty \(T\subseteq[m]\),
\[
 h_{k;\,\sum_{i\in T}e_i-|T|+1}
 =\sum_{\emptyset\ne J\subseteq T}(-1)^{|J|-1}h_{k;e_J},
\]
where \(e_J=(e_j)_{j\in J}\).
\end{lemma}

\begin{proof}
Substitute the definition of \(h_{k;e_J}\) and interchange the two subset sums.  The coefficient of the term indexed by \(L\subseteq T\) is
\[
 \sum_{H\subseteq T\setminus L}(-1)^{|H|},
\]
which vanishes unless \(L=T\), in which case it is equal to one.
\end{proof}


We first recall the ordinary Virasoro recursion.

\begin{proposition}[Ordinary Virasoro recursion]\label{prop:ordinary-corrected}
For \(r,s\ge0\), set
\[
\begin{aligned}
 \mathcal B_{r,s}^{g}(\be)
 &:=
 \left\langle\tau_r\tau_s\prod_{i\in S}\tau_{e_i}\right\rangle_{g-1}
 +\sum_{I\subseteq S}\sum_{g_1+g_2=g}
 \left\langle\tau_r\prod_{i\in I}\tau_{e_i}\right\rangle_{g_1}
 \left\langle\tau_s\prod_{i\in S\setminus I}\tau_{e_i}\right\rangle_{g_2}.
\end{aligned}
\]
Then, for \(k\ge-1\), \(g\ge0\), and \(\be=(e_i)_{i\in S}\),
\[
\begin{aligned}
\left\langle\tau_{k+1}\prod_{i\in S}\tau_{e_i}\right\rangle_g
=&\frac{1}{(2k+3)!!}
\Bigg[
\sum_{j\in S}h_{k;e_j}
\left\langle\tau_{e_j+k}\prod_{i\ne j}\tau_{e_i}\right\rangle_g
+\frac12
\sum_{\substack{r+s=k-1\\r,s\ge0}}
(2r+1)!!(2s+1)!!\,\mathcal B_{r,s}^{g}(\be)
\Bigg]\\
&+\delta_{k,-1}\delta_{g,0}\delta_{|S|,2}
\prod_{i\in S}\delta_{e_i,0}
+\frac1{24}\delta_{k,0}\delta_{g,1}\delta_{S,\emptyset}.
\end{aligned}
\]
The quadratic sum is empty for \(k=-1,0\).
\end{proposition}

\begin{proof}
Extract coefficients from the ordinary Virasoro equations.  For the string operator
\[
  2L_{-1}=-\partial_{t_0}+\sum_{i\ge1}t_i\partial_{t_{i-1}}+\frac{t_0^2}{2},
\]
the quadratic term produces precisely the exceptional value
\(\langle\tau_0^3\rangle_0=1\).  For
\[
  2L_0=-3\partial_{t_1}+\sum_{i\ge0}(2i+1)t_i\partial_{t_i}+\frac18,
\]
the constant term produces \(\langle\tau_1\rangle_1=1/24\).  For \(k\ge1\), no additional constant term occurs.
\end{proof}


We now transfer the ordinary Virasoro recursion to correlators on
\(\Mbar_{g,\mathcal K}\).  For a finite complex \(\mathcal K\) on \(S\), define
\[
 \mathsf S_{\mathcal K}(\be)
 :=
 \sum_{\substack{\pi\in\Pi_{\mathcal K}(S),\ |\pi|=2\\
                   e_I^\sharp=0\ \text{for every }I\in \pi}}
 (-1)^{|S|-2}.
\]
This is the contribution obtained by reconstructing the ordinary three-point
string constant in the presence of the distinguished heavy \(\tau_0\) insertion.

\begin{theorem}[Heavy Virasoro recursion]\label{thm:heavy-corrected}
Let \(\mathcal K\) be a simplicial complex on a finite set \(S\).  For \(k\ge-1\),
\[
\begin{aligned}
\left\langle\tau_{k+1}\prod_{i\in S}\tau_{e_i;i}\right\rangle_g^{\mathcal K}
=\frac{1}{(2k+3)!!}
\Bigg[&
\sum_{\substack{\emptyset\ne J\subseteq S\\J\in \mathcal K}}
 h_{k;e_J}
\left\langle
 \tau_{\sum_{j\in J}e_j-|J|+1+k;J}
 \prod_{j\notin J}\tau_{e_j;j}
\right\rangle_g^{\mathcal K}
\\
&+\frac12\sum_{\substack{r+s=k-1\\r,s\ge0}}
(2r+1)!!(2s+1)!!
\Bigg(
\left\langle\tau_r\tau_s\prod_{i\in S}\tau_{e_i;i}\right\rangle_{g-1}^{\mathcal K}
\\
&+\sum_{I\subseteq S}\sum_{g_1+g_2=g}
\left\langle\tau_r\prod_{i\in I}\tau_{e_i;i}\right\rangle_{g_1}^{\mathcal K|_I}
\left\langle\tau_s\prod_{i\notin I}\tau_{e_i;i}\right\rangle_{g_2}^{\mathcal K|_{S\setminus I}}
\Bigg)
\Bigg]
\\
&\qquad+
\delta_{k,-1}\delta_{g,0}\,\mathsf S_{\mathcal K}(\be)
+\frac1{24}\delta_{k,0}\delta_{g,1}\delta_{S,\emptyset}.
\end{aligned}
\]
\end{theorem}

\begin{proof}
By the definition of the superscript \(\mathcal K\),
\[
 \left\langle\tau_{k+1}\prod_{i\in S}\tau_{e_i;i}\right\rangle_g^{\mathcal K}
 =\sum_{\pi\in\Pi_{\mathcal K}(S)}(-1)^{|S|-|\pi|}
 \left\langle\tau_{k+1}\prod_{I\in \pi}\tau_{e_I^\sharp}\right\rangle_g.
\]
Apply Proposition \ref{prop:ordinary-corrected} to each summand.  The linear terms are transferred by Lemma \ref{lem:mobius}, while the nonseparating term reconstructs directly.  For the separating term, choosing an admissible partition and then distributing its blocks between the two factors is equivalent to choosing a subset \(I\subseteq S\), together with admissible partitions of \(I\) and \(S\setminus I\); the corresponding signs factor accordingly.

For \(k=-1\), the ordinary exceptional term occurs precisely when the admissible partition \(\pi\) has two blocks and both merged exponents are zero.  Summing the corresponding signs gives \(\mathsf S_{\mathcal K}(\be)\).  For \(k=0\), the ordinary exceptional term occurs only for the empty partition, hence only when \(S=\emptyset\), and contributes \(1/24\).
\end{proof}


We now consider the case where there is a distinguished light marking. Fix \(b\in S\), put \(T=S\setminus\{b\}\), and write \(\mathcal{K}_T=\mathcal{K}|_T\). For a partition \(\pi\) of a finite set \(A\), write
\[
\codim(\pi):=|A|-|\pi|.
\]

\begin{lemma}[Heavy--light comparison]\label{lem:heavy-light-corrected}
For every \(k\ge-1\),
\[
\begin{aligned}
&\left\langle\tau_{k+1;b}\prod_{i\in T}\tau_{e_i;i}\right\rangle_g^{\mathcal{K}}
=\left\langle\tau_{k+1}\prod_{i\in T}\tau_{e_i;i}\right\rangle_g^{\mathcal{K}_T}
-
\sum_{\substack{\emptyset\ne I\subseteq T\\I\cup\{b\}\in \mathcal{K}}}
\left\langle
 \tau_{k+1+\sum_{i\in I}e_i-|I|;I\cup\{b\}}
 \prod_{j\in T\setminus I}\tau_{e_j;j}
\right\rangle_g^{\mathcal{K}}.
\end{aligned}
\]
\end{lemma}

\begin{proof}
Expand every term by reconstruction and fix a final admissible partition of \(\{b\}\cup T\).  If the block containing \(b\) is the singleton \(\{b\}\), the light and heavy terms agree.  Otherwise let \(U\subseteq T\) be the nonempty set of other labels in that block.  The correction term indexed by \(I\subseteq U\), \(I\ne\emptyset\), has reconstruction codimension smaller by \(|I|\).  Its total coefficient is
\[
 -\sum_{\emptyset\ne I\subseteq U}(-1)^{\codim(\pi)-|I|}
 =(-1)^{\codim(\pi)}.
\]
This is exactly the coefficient on the light side.
\end{proof}

\begin{theorem}[Corrected light Virasoro recursion]\label{thm:main-corrected}
With the notation above, for every \(k\ge-1\),
\[
\begin{aligned}
\left\langle\tau_{k+1;b}\prod_{i\in T}\tau_{e_i;i}\right\rangle_g^{\mathcal{K}}
=&-\sum_{\substack{\emptyset\ne I\subseteq T\\I\cup\{b\}\in \mathcal{K}}}
\left\langle
 \tau_{k+1+\sum_{i\in I}e_i-|I|;I\cup\{b\}}
 \prod_{j\in T\setminus I}\tau_{e_j;j}
\right\rangle_g^{\mathcal{K}}
\\
&+\frac{1}{(2k+3)!!}
\Bigg[
\sum_{\substack{\emptyset\ne J\subseteq T\\J\in \mathcal{K}_T}}
 h_{k;e_J}
\left\langle
 \tau_{\sum_{j\in J}e_j-|J|+1+k;J}
 \prod_{j\in T\setminus J}\tau_{e_j;j}
\right\rangle_g^{\mathcal{K}_T}
\\
&\hspace{27mm}+\frac12
\sum_{\substack{r+s=k-1\\r,s\ge0}}
(2r+1)!!(2s+1)!!
\Bigg(
\left\langle\tau_r\tau_s\prod_{i\in T}\tau_{e_i;i}\right\rangle_{g-1}^{\mathcal{K}_T}
\\
&\hspace{27mm}+
\sum_{I\subseteq T}\sum_{g_1+g_2=g}
\left\langle\tau_r\prod_{i\in I}\tau_{e_i;i}\right\rangle_{g_1}^{\mathcal{K}|_I}
\left\langle\tau_s\prod_{i\in T\setminus I}\tau_{e_i;i}\right\rangle_{g_2}^{\mathcal{K}|_{T\setminus I}}
\Bigg)
\Bigg]
\\
&\hspace{27mm} +\delta_{k,-1}\delta_{g,0}\,\mathsf S_{\mathcal{K}_T}((e_i)_{i\in T})
+\frac1{24}\delta_{k,0}\delta_{g,1}\delta_{T,\emptyset}.
\end{aligned}
\]
\end{theorem}

\begin{proof}
Insert the heavy recursion of Theorem \ref{thm:heavy-corrected} into the heavy--light comparison of Lemma \ref{lem:heavy-light-corrected}.  The exceptional terms belong to the heavy correlator on the restricted complex \(\mathcal{K}_T\), so they appear exactly as displayed.
\end{proof}

\begin{example}[Why the string correction is necessary]\label{ex:string-correction}
Let \(S=\{1,2,3\}\), and let \(\mathcal{K}\) have the single nonsingleton face \(\{1,2\}\).  Take \(g=0\), \(k=-1\), and \((e_1,e_2,e_3)=(0,1,0)\).  Reconstruction gives
\[
\begin{aligned}
\left\langle\tau_0\tau_{0;1}\tau_{1;2}\tau_{0;3}\right\rangle_0^{\mathcal{K}}
&=\left\langle\tau_0\tau_0\tau_1\tau_0\right\rangle_0
 -\left\langle\tau_0\tau_0\tau_0\right\rangle_0\\
&=1-1=0.
\end{aligned}
\]
The linear part of the recursion contributes \(1\), from \(J=\{2\}\).  On the other hand, the admissible two-block partition \(\{\{1,2\},\{3\}\}\) contributes
\[
  \mathsf S_{\mathcal{K}}(0,1,0)=(-1)^{3-2}=-1.
\]
The corrected right-hand side is therefore \(1-1=0\), as required.
\end{example}

\begin{corollary}[Initial values and determination]\label{cor:initial}
The reconstructed theory contains the two basic constants
\[
  \left\langle\tau_{0;I_1}\tau_{0;I_2}\tau_{0;I_3}\right\rangle_0^{\mathcal K}=1,
  \qquad
  \left\langle\tau_{1;I}\right\rangle_1^{\mathcal K}=\frac1{24},
\]
whenever the displayed formal correlators are defined.  Every geometric or auxiliary coefficient in the paper is determined directly by the finite partition formula of Theorem \ref{thm:reconstruction} and the ordinary Witten--Kontsevich correlators; no separate, unstated induction on the recursion terms is required.
\end{corollary}

\begin{proof}
In the genus-$0$, $3$-point case, every nonsingleton block has negative merged exponent, so only the discrete partition contributes.  In the genus-$1$, $1$-point case there is again only the discrete partition.  The final assertion is precisely the content of Theorem \ref{thm:reconstruction}.
\end{proof}

\section{A Fock-space potential, heavy Virasoro operators, and KdV}\label{sec:fock}
The purpose of this section is to present Virasoro constraints for correlators on $\Mbar_{g,\mathcal{K}}$ in the form of generating series.

\subsection{Collision systems}
We want to consider generating series of correlators. A fixed finite complex cannot support a generating series with arbitrarily many light markings. We therefore allow an infinite label set, while requiring every individual coefficient to have finite support.

\begin{definition}[Collision system]
A collision system \(\mathcal{K}_\infty\) is a simplicial complex of finite subsets of a countably infinite label set \(\Lambda\). Equivalently, it is a compatible family of finite restrictions
\[
  \mathcal{K}_S=\mathcal{K}_\infty|_S,
  \qquad S\subset\Lambda\text{ finite}.
\]
\end{definition}

Every coefficient considered below has finite label support and is therefore determined by one finite restriction \(\mathcal{K}_S\). Geometric correlators are used when the corresponding finite moduli stack exists; all other auxiliary coefficients are defined by reconstruction.

\subsection{The disjoint-support algebra}

For every nonempty finite face \(I\in \mathcal{K}_\infty\) and \(d\ge0\), introduce a symbol \(t_{d;I}\).  For a finite subset \(F\subset\Lambda\), let \(\mathscr F_{\mathcal{K}_{\infty}}[F]\) be the descendant-degree completion of the \(\mathbb Q\)-vector space with basis the monomials
\[
  t_{e_1;I_1}\cdots t_{e_m;I_m}
\]
for which \(\emptyset\ne I_a\in \mathcal{K}_\infty\), the supports \(I_a\) are pairwise disjoint for \(a\ne b\), and \(I_1\cup\cdots\cup I_m=F\).  Set
\[
  \mathscr F_{\mathcal{K}_{\infty}}
  :=\prod_{\substack{F\subset\Lambda\\F\text{ finite}}}
    \mathscr F_{\mathcal{K}_{\infty}}[F].
\]
Define multiplication by disjoint union:
\[
 (t_{\be;\bI})(t_{\mathbf f;\mathbf J})=
 \begin{cases}
 t_{\be,\mathbf f;\bI,\mathbf J},&
 \bigl(\bigcup I_a\bigr)\cap\bigl(\bigcup J_b\bigr)=\emptyset,\\
 0,&\text{otherwise.}
 \end{cases}
\]
For each fixed finite output support, only finitely many decompositions of that support enter a product.  Hence this coefficientwise product is well-defined for arbitrary \(\Lambda\), and \(\mathscr F_{\mathcal{K}_{\infty}}\) is a commutative associative algebra with unit.

\begin{remark}\label{rem:no-light-derivatives}
Although \(\mathscr F_{\mathcal{K}_{\infty}}\) can be presented as a quotient of a polynomial ring by overlap relations, ordinary partial derivatives in the light symbols do not descend to that quotient.  We therefore make no use of operators \(\partial/\partial t_{d;I}\).  Creation and annihilation operators can be defined in a genuine species formalism, but that additional structure is not needed here.
\end{remark}

Let \(x_0,x_1,\ldots\) be ordinary, repeatable heavy variables.  Define the {\em extended potential}
\[
\begin{aligned}
\widehat F^{\mathcal{K}_{\infty}}(\bx;\bt)
:=\sum_{g\ge0}\sum_{m,n\ge0}\frac{1}{m!n!}
\sum_{\substack{\varnothing\ne I_1,\ldots,I_m\in \mathcal{K}_\infty\\
                  I_a\cap I_b=\emptyset\quad(a\ne b)}}
\sum_{\be,\mathbf{k}}
&\left\langle
 \prod_{\alpha=1}^n\tau_{k_\alpha}
 \prod_{a=1}^m\tau_{e_a;I_a}
\right\rangle_g^{\mathcal{K}_{\infty}}
\cdot
\prod_{a=1}^m t_{e_a;I_a}
\prod_{\alpha=1}^n x_{k_\alpha}.
\end{aligned}
\]
This lies in \(\mathscr F_{\mathcal{K}_{\infty}}[\![x_0,x_1,\ldots]\!]\).

\subsection{Cumulants and the explicit potential}

For \(k\ge0\), set
\[
\begin{aligned}
 \Theta_k
 :=\sum_{m\ge1}\frac{(-1)^{m-1}}{m!}
 \sum_{\substack{\varnothing\ne I_1,\ldots,I_m\in \mathcal{K}_\infty\\
                  I_a\cap I_b=\emptyset\quad(a\ne b),\\
                  I_1\cup\cdots\cup I_m\in \mathcal{K}_\infty}}
 \sum_{\substack{e_1,\ldots,e_m\ge0\\
                  e_1+\cdots+e_m-m+1=k}}
 \prod_{a=1}^m t_{e_a;I_a}.
\end{aligned}
\]
Each coefficient has finite support, so this is a well-defined element of \(\mathscr F_{\mathcal{K}_{\infty}}\).

Let \(F^{\WK}(y_0,y_1,\ldots)\) be the ordinary Witten--Kontsevich potential.

\begin{theorem}[Fock-space reconstruction of the potential]\label{thm:fock-potential}
In \(\mathscr F_{\mathcal{K}_{\infty}}[\![\bx]\!]\),
\[
 {
 \widehat F^{\mathcal{K}_{\infty}}(\bx;\bt)
 =F^{\WK}(x_0+\Theta_0,x_1+\Theta_1,x_2+\Theta_2,\ldots).}
\]
\end{theorem}

\begin{proof}
Fix a monomial with pairwise disjoint light supports and a heavy monomial.  In the right-hand side, expanding the Witten--Kontsevich potential and then the \(\Theta_k\)'s amounts to choosing a partition of the light insertions.  A block \(L\) is permitted exactly when the union of its supports is a face of \(\mathcal{K}_\infty\); it contributes exponent
\[
  \sum_{a\in L}e_a-|L|+1
\]
and sign \((-1)^{|L|-1}\).  The product of block signs is \((-1)^{m-|p|}\).  Heavy variables are selected directly from the \(x_k\)'s and are never merged.  The resulting coefficient is exactly the reconstruction formula of Theorem \ref{thm:reconstruction}.

The disjoint-support multiplication is essential: for example, \(t_{0;I}^3=0\), so the cubic term of \(F^{\WK}\) creates no forbidden repeated use of the same nonempty label support.
\end{proof}

\subsection{Heavy Virasoro constraints}

Write \(y_i=x_i+\Theta_i\).  Since \(\Theta_i\) is independent of the heavy variables, \([\partial_{x_i},y_j]=\delta_{ij}\).  Define operators on \(\mathscr F_{\mathcal{K}_{\infty}}[\![\bx]\!]\) by
\[
  2\mathbb L_{-1}
  :=-\frac{\partial}{\partial x_0}
    +\sum_{i\ge1}y_i\frac{\partial}{\partial x_{i-1}}
    +\frac{y_0^2}{2},
\]
\[
  2\mathbb L_0
  :=-3\frac{\partial}{\partial x_1}
    +\sum_{i\ge0}(2i+1)y_i\frac{\partial}{\partial x_i}
    +\frac18,
\]
and, for \(k\ge1\),
\[
\begin{aligned}
 2\mathbb L_k
 &: =-(2k+3)!!\frac{\partial}{\partial x_{k+1}}
 +\sum_{i\ge0}h_{k;i}y_i\frac{\partial}{\partial x_{i+k}}
 +\frac12
 \sum_{\substack{r+s=k-1\\r,s\ge0}}
 (2r+1)!!(2s+1)!!
 \frac{\partial^2}{\partial x_r\partial x_s}.
\end{aligned}
\]

\begin{theorem}[Heavy Virasoro operators]\label{thm:heavy-operators}
For all \(k\ge-1\),
\[
  \mathbb L_k\bigl(e^{\widehat F^{\mathcal{K}_{\infty}}}\bigr)=0,
\]
and
\[
  [\mathbb L_m,\mathbb L_n]=(m-n)\mathbb L_{m+n}.
\]
\end{theorem}

\begin{proof}
By Theorem \ref{thm:fock-potential}, \(\widehat F^{\mathcal{K}_{\infty}}=F^{\WK}(\by)\).  The operators \(\mathbb L_k\) are exactly the ordinary Witten--Kontsevich Virasoro operators in the translated variables \(y_i=x_i+\Theta_i\), with \(\partial_{y_i}=\partial_{x_i}\).  Both assertions therefore follow from the ordinary Virasoro constraints and commutation relations.
\end{proof}

\begin{remark}
The string operator contains
\[
  \frac12(x_0+\Theta_0)^2
  =\frac{x_0^2}{2}+x_0\Theta_0+\frac{\Theta_0^2}{2}.
\]
The mixed and purely heavy terms are indispensable.  This is precisely what is lost if one writes only a quadratic expression in the light variables.
\end{remark}

\subsection{The heavy two-point series and the KdV hierarchy}

Define
\[
 U^{\mathcal{K}_{\infty}}(\bx;\bt)
 :=\frac{\partial^2\widehat F^{\mathcal{K}_{\infty}}}{\partial x_0^2}.
\]
The derivatives are legitimate because the heavy variable \(x_0\) is repeatable.

Let the Gelfand--Dickey polynomials be normalized by
\[
 R_0[U]=U,
 \qquad
 R_n[0]=0, \quad n\ge0,
\]
and, for \(n\ge0\), by
\[
 \frac{\partial R_{n+1}}{\partial y_0}
 =\frac1{2n+3}
 \left(
  U_{y_0}+2U\frac{\partial}{\partial y_0}
  +\frac14\frac{\partial^3}{\partial y_0^3}
 \right)R_n[U].
\]
The condition \(R_n[0]=0\) fixes the integration constant at every step.  With this normalization,
\[
 R_1[U]=\frac12U^2+\frac1{12}U_{y_0y_0},
\]
so the first nontrivial flow is
\[
 U_{y_1}=UU_{y_0}+\frac1{12}U_{y_0y_0y_0}.
\]

\begin{theorem}[Explicit heavy KdV solution]\label{thm:kdv}
Let
\[
 U^{\WK}(\by):=\frac{\partial^2F^{\WK}}{\partial y_0^2}.
\]
Then
\[
 {
 U^{\mathcal{K}_{\infty}}(\bx;\bt)
 =U^{\WK}(x_0+\Theta_0,x_1+\Theta_1,x_2+\Theta_2,\ldots).}
\]
Consequently, for every \(n\ge0\),
\[
 \frac{\partial U^{\mathcal{K}_{\infty}}}{\partial x_n}
 =\frac{\partial}{\partial x_0}R_n[U^{\mathcal{K}_{\infty}}].
\]
On the translated initial slice \(y_i=x_i+\Theta_i=0\) for \(i>0\), one has \(U^{\mathcal{K}_{\infty}}=y_0=x_0+\Theta_0\).
\end{theorem}

\begin{proof}
Differentiate the identity in Theorem \ref{thm:fock-potential} twice with respect to the repeatable heavy variable \(x_0\).  Since \(\Theta_i\) is independent of every \(x_j\), heavy differentiation is the same as differentiation in the translated times \(y_i\).  The ordinary Witten--Kontsevich KdV hierarchy then gives the result.
\end{proof}

\subsection{The Hassett specialization without infinite coefficients}

Let \(A\subset[0^+,1]\) be additively closed and take \(\Lambda=A\times\mathbb N\) with
\[
 I\in \mathcal{K}_\infty^A
 \quad\Longleftrightarrow\quad
 \sum_{(a,r)\in I}a\le1.
\]
One must not identify the infinitely many standardized label variables of the same weight.  Instead, set all nonsingleton cluster variables to $0$ and apply the coefficientwise symmetrization
\[
\begin{aligned}
 \operatorname{Sym}_A(\Phi)
 :=\sum_{n\ge0}\frac1{n!}
 \sum_{\substack{a_1,\ldots,a_n\in A\\e_1,\ldots,e_n\ge0}}
 &\text{Coeff}_{\left[
 \prod_{j=1}^n t_{e_j;\{(a_j,j)\}}
 \right]}(\Phi)
 \cdot\prod_{j=1}^n z_{e_j;a_j}.
\end{aligned}
\]
Here \(\text{Coeff}_{[M]}(\Phi)\) denotes the coefficient of \(M\), and the displayed sum is understood coefficientwise in the product over finite multisets of pairs \((e,a)\). For any fixed output monomial, there are only finitely many orderings of its finite multiset of weights.  Every input coefficient uses only the standardized finite label set \(\{(a_j,j):1\le j\le n\}\), so no infinite sum occurs.

Let \(\widehat F^{\mathcal{K}_{\infty}^A}_{\mathrm{sing}}\) denote the result of setting every nonsingleton cluster variable to $0$. Then the precise specialization is
\[
 \operatorname{Sym}_A\!\left(
  \widehat F^{\mathcal{K}_{\infty}^A}_{\mathrm{sing}}
 \right)\Big|_{\bx=0}
 =F^A(\bz).
\]
If \(1\in A\), one may retain the heavy variables; in that case
\[
 \operatorname{Sym}_A\!\left(
  \widehat F^{\mathcal{K}_{\infty}^A}_{\mathrm{sing}}
 \right)(\bx;\bz)
 =F^A(\widetilde\bz),
 \qquad
 \widetilde z_{d;a}=
 \begin{cases}
   z_{d;a},&a\ne1,\\
   z_{d;1}+x_d,&a=1.
 \end{cases}
\]
Indeed, a heavy insertion and a weight-one light singleton are both noncolliding weight-one markings, and the exponential generating convention combines them by the binomial formula.  If \(1\notin A\), retaining \(\bx\) instead gives the analogous potential for \(A\cup\{1\}\), with weight-$1$ variable \(x_d\).  This symmetrization is a species-level assembly operation, not a homomorphism obtained by naively identifying infinitely many variables.

\appendix

\section{Natural non-Hassett constructions}\label{sec:examples}

This appendix discusses natural moduli stacks of simplicially stable curves that are not Hassett spaces.

\subsection{The established logarithmic tail-function construction}

A general source of non-threshold spaces is the logarithmic contraction developed by Blankers and Bozlee.

\begin{theorem}[Blankers--Bozlee tail-function contraction]\label{thm:BB-tail}
Let \(g\ge0\) and \(n\ge1\), and let \(\mathcal K\) be a simplicial complex on \([n]\), at least triparted when \(g=0\).  On the universal stable curve with its basic logarithmic structure, there is a universal tail function \(\mu_{\mathcal K}\) whose support on every geometric fiber is the union of the rational tails whose marking sets belong to \(\mathcal K\).  Twisting the logarithmic dualizing bundle by \(\mu_{\mathcal K}\) and taking a relative Proj produces a base-change-compatible contraction
\[
 \rho_{\mathcal K}:\Mbar_{g,n}\longrightarrow\Mbar_{g,\mathcal K}
\]
that contracts exactly those tails.
\end{theorem}

\begin{proof}
This is a summary of \cite[Sections~5 and 6]{BB}. Blankers--Bozlee define tail functions as piecewise-linear functions on tropicalizations, prove their compatibility with face contractions, and identify universal tail functions with extremal assignments supported on rational tails, hence with collision complexes.  For a family \(\pi:C\to S\), the associated line bundle is
\[
  \mathcal L_{\mathcal K}
  =\omega_{C/S}(\Sigma)\otimes\mathcal O_C(\mu_{\mathcal K}).
\]
It is trivial on the selected rational tails and positive on their complement.  Their cohomology-and-base-change argument shows that
\[
  \operatorname{Proj}_S
  \bigoplus_{m\ge0}\pi_*\mathcal L_{\mathcal K}^{\otimes m}
\]
contracts precisely the support of \(\mu_{\mathcal K}\) and commutes with base change.  Applied universally, this gives \(\rho_{\mathcal K}\) and the family of \(\mathcal K\)-stable curves.  No separate claim about a logarithmic structure or universal property on \(\Mbar_{g,\mathcal K}\) is needed here.
\end{proof}

\subsection{A concrete genus-zero non-threshold example}

Let
\[
 O=\{1,3,5\},\qquad E=\{2,4,6\},
\]
and define a complex on the set \([7]\) by
\[
 \mathcal K^\dagger:=2^O\cup2^E\cup\{\emptyset,\{7\}\}.
\]
Thus the odd triple may collide, the even triple may collide, and no odd label may collide with an even label or with the seventh marking.

\begin{proposition}\label{prop:Kdagger}
The complex \(\mathcal K^\dagger\) is at least triparted and is not a threshold complex.  Hence \(\Mbar_{0,\mathcal K^\dagger}\) is a natural modular compactification produced by the logarithmic tail-function contraction of Theorem \ref{thm:BB-tail}, but it is not a Hassett space.
\end{proposition}

\begin{proof}
Every \(\mathcal K^\dagger\)-partition has the singleton block \(\{7\}\), at least one block containing odd labels, and at least one block containing even labels.  It therefore has at least three blocks.

Suppose \(\mathcal K^\dagger\) were defined by weights \(a_1,\ldots,a_7\).  Since \(O\) and \(E\) are faces,
\[
 a_1+a_3+a_5\le1,
 \qquad
 a_2+a_4+a_6\le1.
\]
Some odd weight and some even weight are therefore at most \(1/3\).  Their sum is at most \(2/3\), so the corresponding odd--even pair would be a face of the threshold complex, contrary to the definition of \(\mathcal K^\dagger\).

By Theorem \ref{thm:BB-tail}, the universal logarithmic tail function contracts exactly the rational tails supported entirely on a face of \(\mathcal K^\dagger\).  The resulting Deligne--Mumford moduli stack is \(\Mbar_{0,\mathcal K^\dagger}\).  Since \(\mathcal K^\dagger\) is not a threshold complex, \cite[Corollary~4.22]{BB} implies that \(\Mbar_{0,\mathcal K^\dagger}\) is not a Hassett space.
\end{proof}

\end{document}